\documentclass[12pt,a4paper]{amsart}
\usepackage{amsfonts}
\usepackage[top=35mm, bottom=35mm, left=30mm, right=30mm]{geometry}
\usepackage[colorlinks=true,citecolor=blue]{hyperref}
\usepackage{mathptmx}
\usepackage{eucal}
\usepackage{graphicx}
\usepackage{mathrsfs}
\usepackage{amssymb}
\usepackage{amsmath}
\usepackage{amsthm}
\usepackage{tikz,pgfplots,float}

\usepackage{xcolor}
\newtheorem{thm}{Theorem}[section]
\newtheorem{cor}[thm]{Corollary}
\newtheorem{lem}[thm]{Lemma}
\newtheorem{prop}[thm]{Proposition}

\theoremstyle{definition}
\newtheorem{defn}[thm]{Definition}
\newtheorem{exam}[thm]{Example}
\newtheorem{rem}[thm]{Remark}
\numberwithin{equation}{section}

\newcommand{\N}{\mathbb{N}}
\newcommand{\calF}{\mathcal{F}}
\newcommand{\calA}{\mathcal{A}}
\newcommand{\calB}{\mathcal{B}}

\newcommand{\calC}{\mathcal{C}}
\newcommand{\calS}{\mathcal{S}}
\newcommand{\scrS}{\mathscr{S}}
\newcommand{\scrB}{\mathscr{B}}

\newcommand{\w}{\omega}
\newcommand{\normmm}[1]{{\left\vert\kern-0.25ex\left\vert\kern-0.25ex\left\vert #1
   \right\vert\kern-0.25ex\right\vert\kern-0.25ex\right\vert}}

\usepackage{tkz-graph}
\GraphInit[vstyle = Shade]
\usepackage{nicematrix}
\begin{document}
%\linenumbers%xian-shi-hang-hao-xu-yao-lineno-package
\title[Invariance entropy for uncertain control systems]{Invariance entropy for uncertain control systems}
\author{Xingfu Zhong}
\address{School of Mathematics and Statistics, Guangdong University of Foreign Studies\\
Guangzhou, 510006 P. R. China}
\email{xfzhong@gdufs.edu.cn}

\author{Yu Huang}
\address{School of Mathematics, Sun Yat-Sen University,
GuangZhou 510275, P. R. China}
\email{stshyu@mail.sysu.edu.cn}

\author{Xingfu Zou}
\address{Department of Applied Mathematics, University of Western Ontario London, Ontario, N6A 5B7, Canada}
\email{xzou@uwo.ca}


\begin{abstract}
We introduce a notion of invariance entropy for \emph{uncertain} control systems, which is, roughly speaking, the exponential growth rate of ``branches'' of ``trees'' that are formed by controls and are necessary to achieve invariance of controlled invariant subsets of the state space. This entropy extends the invariance entropy for \emph{deterministic} control systems introduced by Colonius and Kawan (2009). We show that invariance feedback entropy, proposed by Tomar, Rungger, and Zamani (2020), is bounded from below by our invariance entropy. We generalize the formula for the calculation of entropy of invariant partitions obtained by Tomar, Kawan, and Zamani (2020) to quasi-invariant-partitions. Moreover, we also derive lower and upper bounds for entropy of a quasi-invariant-partition by spectral radii of its adjacency matrix and weighted adjacency matrix. With some reasonable assumptions, we obtain explicit formulas for computing invariance entropy for \emph{uncertain} control systems and invariance feedback entropy for finite controlled invariant sets.
\end{abstract}

%\date{\today}

\maketitle

%\tableofcontents
\section{Introduction}
Entropy for a dynamical system is an intrinsic quantity that measures the complexity  of the system. There are two popular dynamical entropies: one is measure-theoretic entropy and the other is topological entropy. The former was introduced by one of the most
influential mathematicians of modern times, Kolmogorov~\cite{Kolmogorov1958A},
and improved by his student Sinai~\cite{Sina1959On} who practically brought it to the contemporary form; the latter was proposed via open covers by Adler et al.~\cite{Adler1965} and was redefined by Dinaburg~\cite{Dinaburg1970} and
Bowen~\cite{Bowen1971} independently in the language of metric spaces. These two definitions of dynamical entropies resemble the definition of Shannon's entropy~\cite{shannon1948mathematical}. Both of them measure the exponential rates of growth of numbers of orbits in some sense:
\begin{itemize}
  \item measure-theoretic entropy counts the number of typical $n$-orbits, while
  \item topological entropy counts all distinguishable $n$-orbits.
\end{itemize}
Note that topological entropy is equal to the supremum of measure-theoretic entropy over all invariant measures (this basic relationship between topological entropy and measure-theoretic entropy is called variational principle). We refer the reader to references~\cite{Downarowicz2011Entropy,Young2003Entropy,Katok2007Fifty} for more details about the history of dynamical entropies.

Invariance (or stabilization) is another important notion that described a widely needed property for control systems. When a control system involves communication of information, an interesting question involving invariance is how much information practically needs to be communicated between the coder and controller in order to make a given set invariant under information constraint(s). Early work on this topics are Delchamps~\cite{Delchamps1990Stabilizing} and Wong and Brockeet~\cite{Wong-Brockett1999}, in which they respectively investigated quantized feedbacks and the influence of restricted communication channels for stabilization. The seminal work by Nair et al.~\cite{Nair2004Topological} first addressed the problem of
data-rate-limited stabilization by introducing topological feedback entropy. This entropy characterizes the smallest average data rate at which a subset of the state space can be made invariant and it resembles the topological entropy by using similar
open cover techniques. Then Colonius and Kawan~\cite{Colonius-Kawan2009} introduced an invariance entropy to describe the exponential growth rate of different control functions sufficient for making a subset of the state space invariant. The definition of this invariance entropy is analogous to that of the topological entropy by Dinaburg~\cite{Dinaburg1970} and Bowen~\cite{Bowen1971} by replacing distinguishable orbits by different control functions. Colonius et al.~\cite{Colonius2013A} showed that this invariance entropy and the topological feedback entropy are equivalent with some suitable modifications.

For better understanding of invariance entropy, various notions relating to invariance entropy have been proposed by several groups of researchers from different views, such as invariance pressure~\cite{Colonius2018PressureA,Colonius2018PressureB,
Colonius2019Bounds,
Zhong2020Variational,Zhong2018Invariance,Colonius2021Controllability}, measure-theoretic versions of invariance entropy~\cite{Colonius2018Invariance,Colonius2018Metric,Wang2019MeasureInvariance,Wang-Huang2022Inverse},
dimension types of invariance entropy~\cite{Huang2018Caratheodory}, complexity of invariance entropy~\cite{Wang2019dichotomy,Zhong2020Equi,Zhong2022Equi}. Note that Kawan and Y\"{u}ksel~\cite{Kawan-Yuksel2019Stabilization-entropy} introduced a notion of stabilization entropy which is a variant of invariance entropy. We refer the reader to the monograph written by Kawan~\cite{Kawan2013} for more details about invariance entropy of \emph{deterministic} control systems and to~\cite{savkin2006analysis,Liberzon2018Entropy,Kawan-Matveev-Pogromsky2021Remote} and the references therein for observability that is another data-rate-limited task and is closely related to controlled invariance.

In the context of \emph{uncertain} control systems, invariance feedback entropy (IFE) was introduced by Rungger and Zamani~\cite{Rungger2017Invariance} to quantify the state information required by any controller to render a subset of the state space invariant. Later, Tomar et al.~\cite{Tomar2017invariance} and Tomar and Zamani~\cite{Tomar2020Compositional} further investigated the properties of IFE. Recently, Tomar et al.~\cite{Tomar2020Numerical} presented wonderful algorithms for the numerical computation of invariance entropy for \emph{deterministic} control systems and IFE for \emph{uncertain} control systems respectively. Particularly, they showed that the entropy with respect to an invariant partition is equivalent to the maximum mean cycle weight (MMCW) of the weighted graph associated with this partition. Their algorithms allow us to compute upper bounds for invariance entropy of \emph{deterministic} control systems and IFE of \emph{uncertain} control systems. From the above, there arise the following questions naturally:
\begin{itemize}
  \item[Q1.] Whether or not there is an analogous version of the invariance entropy introduced by Colonius and Kawan via controls for \emph{uncertain} control systems? Note that the definition of IFE (see Subsection~\ref{subsec:ife}) begins with a cover. If such a version exists, how is the relation between this invariance entropy and IFE for \emph{uncertain} control systems?
  \item[Q2.] If the formula for the calculation of entropy of an invariant partition, obtained by Tomar et al.~\cite{Tomar2020Numerical}, also holds for some ``weak'' invariant partitions?
  \item[Q3.] Which conditions guarantee the existence of ``generators''? By a generator we mean an invariant cover such that IFE is equal to the entropy of this cover.
\end{itemize}

In order to answer these questions, we introduce a new notion of invariance entropy via control functions for \emph{uncertain} control systems. Roughly speaking, our invariance entropy for \emph{uncertain} control systems is the exponential growth rate of ``branches'' of ``trees''. Such trees are formed by control functions that are necessary to make the target set invariant. We emphasize that our notion of invariance entropy discussed in the sequel is for \emph{uncertain} control system and that Colonius and Kawan have used the term ``invariance entropy'' for \emph{deterministic} control system.

The structure of the paper is as follows. In Section~\ref{sec:invariance-entropy}, we introduce the concept of invariance entropy for \emph{uncertain} control system, give some basic properties for this invariance entropy, and show that this invariance entropy is less than or equal to the invariance feedback entropy (an answer to Q1, see Theorem~\ref{thm:inv-less-ifb}). It is worthy to note that invariance entropy is equal to topological feedback entropy for \emph{deterministic} controls systems, see~\cite{Colonius2013A} or~\cite{Kawan2013}. In Section~\ref{sec:calculation}, we derive some formulas for the calculation of our invariance entropy and IFE. We show that the invariance entropy for a controlled invariant set equals to the
logarithm of the spectral radius of its admissible matrix under some technical assumption (see Theorem~\ref{thm:symbolic-log-radius}). We also extend the formula for the calculation of entropy of invariant partitions, obtained by Tomar et al.~\cite{Tomar2020Numerical}, to quasi-invariant-partitions; that is, the entropy for a quasi-invariant-partition is equal to
the maximal mean weight of this partition (an answer to Q2, see Theorem~\ref{thm:invariant-partition-log-weight}), and show that if the spectral radius of the adjacency matrix of this partition is $1$, then the entropy for this partition equals to the logarithm
of the spectral radius of its weighted adjacency matrix (see
Theorem~\ref{thm:upper-lower-bounds-partition}). Finally, we show that there exists generators for a finite controlled invariant set; i.e., IFE for a finite controlled invariant set is equal to the entropy of its atom partition (an answer to Q3, see Theorem~\ref{thm:IFE-equi-atom-partitions}).

\section{Invariance entropy}\label{sec:invariance-entropy}
This section consists of two subsections. The first presents some basic properties for invariance entropy. The second recalls the definition of invariance feedback entropy and shows that invariance entropy is a tight lower bound for invariance feedback entropy.

\subsection{Invariance entropy}
First let us introduce terminology and notation. (We borrow some terms and signs from~\cite{Colonius-Kawan2009}.) We use $f:A\rightrightarrows B$ to denote a \emph{set-valued map} from $A$ into $B$, whereas $f:A\rightarrow B$ denotes an ordinary map. If $f$ is set-valued, then $f$ is \emph{strict} if for every $a\in A$ we have $f(a)\neq\emptyset$. The composition of $f:A\rightrightarrows B$ and $g:C\rightrightarrows A$, $(f\circ g)(x)=f(g(x))$ is denoted by $f\circ g$. We call a triple $\Sigma:=(X,U,F)$ a \emph{system} if
$X$ and $U$ are nonempty sets and $F: X\times U\rightrightarrows X$ is strict.
Recall that $Q\subset X$ is called \emph{controlled invariant} with respect to a system $\Sigma=(X,U,F)$, if for every $x\in Q$ there exists $u\in U$ such that $F(x,u)\subset Q$. Fixing $u\in U$ and $Q\subset X$, let $Q_u=\{x\in Q| F(x,u)\subset Q\}$.

A subset $S\subset U^n$ is said to be \emph{an admissible family of length $n$} for $Q$ if
\begin{itemize}
  \item[(a).] $\omega'_0=\omega''_0$ for any $\omega',\omega''\in S$;
  \item[(b).] there exists $x\in Q$ such that for any $\omega\in S$,
\begin{align*}
&F(I_\omega^i(x),\omega_i)\subset\bigcup_{\substack{\omega'\in S,\\ \omega'_{[0,i]}=\omega_{[0,i]}}}Q_{\omega'_{i+1}}, \\
&I_{\omega}^{i+1}(x)=F(I_\omega^{i}(x),\omega_i)\cap Q_{\omega_{i+1}}\neq\emptyset,\forall~i=0,1,\ldots,n-2,\\
&I_{\omega}^{n}(x)=F(I_\omega^{n-1}(x),\omega_{n-1})\subset Q,
\end{align*}
where $I_\omega^{0}(x)=x$.
\end{itemize}

Let
\[AF^n(Q)=\{S\subset U^n: S~\text{is an admissible family for}~Q\}\]
and
\[AF(Q)=\bigcup_{n=1}^\infty AF^n(Q).\]
Given $S\in AF(Q)$, the set of points that satisfy condition (b) is denoted by $Q_S$.

Let $K\subset Q$ be a nonempty set. A set $\mathscr{S}\subset U^n$ is called an \emph{$(n,K,Q)$-spanning set} of $(K,Q)$ if
\[K\subset \bigcup_{S\subset\mathscr{S}~\text{and}~S\in AF^n(Q)}Q_S.\]

We use the notation $r_{inv}(n,K,Q)$ for the minimal number of a spanning set, i.e.,
\[r_{inv}(n,K,Q):=\inf\{\sharp \scrS:\scrS~\text{is an}~(n,K,Q)\text{-spanning set of}~(K,Q)~\},\]
where $\sharp\scrS$ denotes the cardinality of $\scrS$. For convenience, we write $r_{inv}(n,Q)$ in place of $r_{inv}(n,Q,Q)$.

\begin{defn}
Given a pair $(K,Q)$, we define the \emph{invariance entropy} of $(K,Q)$ by
\[h_{inv}(K,Q)=h_{inv}(K,Q;\Sigma):=\limsup_{n\to\infty}\frac{\log  r_{inv}(n,K,Q)}{n},\]
where $\log$ signifies the logarithm base $2$.
\end{defn}

\begin{rem}
When $F$ is a single-valued map, the definition of invariance entropy coincides with that of invariance entropy for \emph{deterministic} control systems (see~\cite[Definition 2.2]{Kawan2013}).
\end{rem}

Recall that a sequence of real numbers $\{a_n\}_{n\geq1}$ is \emph{subadditive} if $a_{n+p}\leq a_n+a_p$ for all $n,p\in\N$.

\begin{lem}\label{lem:subadditive}(\cite[Theorem 4.9]{Walters1982} or \cite[Lemma B.3]{Kawan2013})
If $\{a_n\}_{n\geq1}$ is a subadditive sequence, then $\lim_{n\to\infty}\frac{a_n}{n}$ exists and equals $\inf_{n}\frac{a_n}{n}$.
\end{lem}

The rest of this subsection will generalize some properties of invariance entropy for \emph{deterministic} control systems (see~\cite{Kawan2013}) to \emph{uncertain} control systems, including finiteness, time discretization, finite stability, and invariance under conjugacy.

\begin{prop}\label{prop:subadditive-of-invariance-entropy}
Let $\Sigma=(X,U,F)$ be a system and $Q\subset X$ be a controlled invariant set. Then the following assertions hold:
\begin{enumerate}
  \item The number $r_{inv}(n,Q)$ is either finite for all $n\in\N$ or for none.
  \item The function $n\mapsto\log r_{inv}(n,Q)$, $\N\to[0,+\infty]$, is subadditive and thus
\[h_{inv}(Q)=\lim_{n\to\infty}\frac1{n}\log r_{inv}(n,Q).\]
\end{enumerate}
\end{prop}

\begin{proof}
(1) Suppose there exists $N\in\N$ such that $r_{inv}(N,Q)<\infty$. It is easy to check that $r_{inv}(n,Q)\leq r_{inv}(N,Q)$ for every $n<N$. We now show that $r_{inv}(n,Q)< \infty$ for every $n>N$. Given $n\geq N$, pick $k\in\N$ such that $kN>n$. Let $\scrS=\{\omega^1,\ldots,\omega^m\}$ be a minimal $(N,Q)$-spanning set, where $m=r_{inv}(N,Q)$, and let
\[\scrS_k=\{\omega\in U^{kN}:~\forall ~0\leq i\leq k-1 ~\exists~ \omega'\in \scrS~\text{s.t.}~\omega_{[Ni,(i+1)N)}=\omega'\}.\]
We shall show that $\scrS_k$ is a $(kN,Q)$-spanning set. For every $x\in Q$ there exists $S_x^1\subset \scrS$ such that $\omega'_0=\omega''_0$ for any $\omega',\omega''\in S_x^1$ and for any $\omega\in S_x^1$,
\begin{align*}
&F(I_\omega^i(x),\omega_i)\subset\bigcup_{\substack{\omega'\in S_x^1,\\ \omega'_{[0,i]}=\omega_{[0,i]}}}Q_{\omega'_{i+1}}, \\
&I_{\omega}^{i+1}(x)=F(I_\omega^{i}(x),\omega_i)\cap Q_{\omega_{i+1}}\neq\emptyset,\forall~i=0,1,\ldots,N-2,\\
&I_{\omega}^{N}(x)=F(I_\omega^{n-1}(x),\omega_{N-1})\subset Q,
\end{align*}
where $I_\omega^{0}(x)=x$.
For every $y\in I_{\omega}^{N}(x)\subset Q$ there exists $S_{\omega,y}\subset \scrS$ such that $\omega'_0=\omega''_0$ for any $\omega',\omega''\in S_{\omega,y}$ and for any $\bar{\omega}\in S_{\omega,y}$,
\begin{align*}
&F(I_{\bar{\omega}}^i(y),{\bar{\omega}}_i)\subset\bigcup_{\substack{{\bar{\omega}}'\in S_{\omega,y},\\ {\bar{\omega}}'_{[0,i]}={\bar{\omega}}_{[0,i]}}}Q_{{\bar{\omega}}'_{i+1}}, \\
&I_{{\bar{\omega}}}^{i+1}(y)=F(I_{\bar{\omega}}^{i}(y),{\bar{\omega}}_i)\cap Q_{{\bar{\omega}}_{i+1}}\neq\emptyset,\forall~i=0,1,\ldots,N-2,\\
&I_{{\bar{\omega}}}^{N}(y)=F(I_{\bar{\omega}}^{N-1}(y),{\bar{\omega}}_{N-1})\subset Q,
\end{align*}
where $I_{\bar{\omega}}^{0}(y)=y$.
Let $\bar{S}_x^2=\cup_{\omega\in S_x^1}\cup_{y\in I_{\omega}^{N}(x)}S_{\omega,y}$ and
\[S_x^2:=\{\omega\in U^{2N}:\exists \omega'\in S_x^1,\omega''\in\bar{S}_x^2~\text{s.t.}~\omega_{[0,N)}=\omega'~\text{and}
~\omega_{[N,2N)}=\omega''\}\subset \scrS_2.
\]
Hence we have $x\in Q_{S_x^2}$. Repeating the process $k$ times, we can obtain $S_x^k\subset \scrS_k$ such that $x\in Q_{S_x^k}$.
This means that $\scrS_k$ is a $(kN,Q)$-spanning set.

(2) By Lemma~\ref{lem:subadditive}, we shall prove that
\[r_{inv}(n+p,Q)\leq r_{inv}(n,Q)\cdot r_{inv}(p,Q),~\forall~n,p\in\N.\]
Assume that $\scrS_1$ is a minimal $(n,Q)$-spanning set and $\scrS_2$ is a minimal $(p,Q)$-spanning set. Let
\[\scrS=\{\omega\in U^{n+p}:\exists~\omega'\in \scrS_1,~\omega''\in \scrS_2~\text{s.t.}~\omega_{[0,n)}=\omega'~\text{and}~\omega_{[n,n+p-1)}=\omega''\}.\]
Similar to the proof of (1), we can show that $\scrS$ is an $(n+p,Q)$-spanning set. So
\[r_{inv}(n+p,Q)\leq r_{inv}(n,Q)\cdot r_{inv}(p,Q),~\forall~n,p\in\N,\]
which completes the proof.
\end{proof}

\begin{rem}
We see from Proposition~\ref{prop:subadditive-of-invariance-entropy} that $r_{inv}(n,Q)$ is finite for some $n$ if and only if $r_{inv}(n,Q)$ is finite for all $n$ if and only if $h_{inv}(Q)$ is finite.
\end{rem}

\begin{prop}[Time discretization]\label{prop:time-discretization}
Let $\Sigma=(X,U,F)$ be a system and $Q\subset X$ be a controlled invariant set. If $K\subset Q$ and $m\in\N$, then
\[h_{inv}(K,Q)=\limsup_{n\to\infty}\frac{1}{nm}\log r_{inv}(nm,K,Q).\]
\end{prop}

\begin{proof}
It is clear that $h_{inv}(K,Q)\geq\limsup_{n\to\infty}\frac{1}{nm}\log r_{inv}(nm,K,Q)$. We now show the converse inequality. By definition of $h_{inv}(K,Q)$, we can take a sequence $\{q_k\}_{k\geq1}$ such that
\[h_{inv}(K,Q)=\lim_{k\to\infty}\frac{1}{q_k}\log r_{inv}(q_k,K,Q).\]
For every $k\geq m$, there exists $n_k\geq1$ such that $n_km\leq q_k\leq (n_k+1)m$, and $n_k\to\infty$ for $k\to\infty$. Then we have $r_{inv}(q_k,K,Q)\leq r_{inv}((n_k+1)m,K,Q)$.
It follows that
\[\frac{1}{q_k}\log r_{inv}(q_k,K,Q)\leq \frac{1}{n_km}\log r_{inv}((n_k+1)m,K,Q).\]
It follows by a straightforward computation that
\begin{align*}
h_{inv}(K,Q)\leq \limsup_{n\to\infty} \frac{1}{nm}\log r_{inv}(nm,K,Q).
\end{align*}
\end{proof}

\begin{prop}[Subsets rule or finite stability]\label{prop:subset-rule}
Let $\Sigma=(X,U,F)$ be a system and $Q\subset X$ be a controlled invariant set. If $K\subset Q$ and $K=\cup_{i=1}^mK_i$. Then $h_{inv}(K,Q)=\max\limits_{i=1,\ldots,m}h_{inv}(K_i,Q)$.
\end{prop}

\begin{proof}
Obviously, we have $h_{inv}(K,Q)\geq\max\limits_{i=1,\ldots,m}h_{inv}(K_i,Q)$. To show the converse inequality, note that $r_{inv}(n,K,Q)\leq\sum_{i=1}^mr_{inv}(n,K_i,Q)$. For every $n$ pick $\hat{K}_{n}\in\{K_1,\ldots,K_m\}$ such that
\[r_{inv}(n,\hat{K}_{n},Q)=\max_{i=1,\ldots,m}r_{inv}(n,K_i,Q).\]
It immediately follows that
\[r_{inv}(n,K,Q)\leq m\cdot r_{inv}(n,\hat{K}_{n},Q).\]
Thus
\[\log r_{inv}(n,K,Q)\leq\log m+\log r_{inv}(n,\hat{K}_{n},Q).\]
Choose $n_k\to\infty$ such that
\[\lim_{k\to\infty}\frac{1}{n_k}\log r_{inv}(n_k,K,Q)=\limsup_{n\to\infty}\frac{1}{n}\log r_{inv}(n,K,Q)\]
and $\hat{K}_{n_k}=K_j$ for some $j\in\{1,2,\ldots,m\}$ and all $k$.
A brief calculation then shows that
\begin{align*}
h_{inv}(K,Q)&=\lim_{k\to\infty}\frac{1}{n_k}\log r_{inv}(n_k,K,Q)\\
&\leq\limsup_{k\to\infty}\frac{1}{n_k}\big(\log m+\log r_{inv}(n_k,K_{j},Q)\big)\\
&\leq\limsup_{k\to\infty}\frac{1}{n_k}\log r_{inv}(n_k,K_{j},Q)\\
&\leq\limsup_{n\to\infty}\frac{1}{n}\log r_{inv}(n,K_{j},Q)\\
&=h_{inv}(K_j,Q)\leq\max_{i=1,\ldots,m}h_{inv}(K_i,Q).
\end{align*}
%\begin{align*}
%h_{inv}(K,Q)\leq\max_{i=1,\ldots,m}h_{inv}(K_i,Q).
%\end{align*}
\end{proof}

Consider two systems $\Sigma_i= (X_i, U_i, F_i)$, $i=1,2$. Let $\pi:X_1\to X_2$ be a continuous map and $r:U_1\to U_2$ a map. We say $(\pi,r)$ is a \emph{semi-conjugacy from $\Sigma_1$ to $\Sigma_2$} if
\[F_2(\pi(x),r(u))\subset\pi(F_1(x,u)),~\forall~x\in X_1,~u\in U_1.\]

\begin{prop}
Let $\Sigma_1= (X_1, U_1, F_1)$ and $\Sigma_2= (X_2, U_2, F_2)$ be two systems, $Q\subset X_1$ be controlled invariant, and $(\pi,r)$ a semi-conjugacy from $\Sigma_1$ to $\Sigma_2$. Then for any $K\subset Q$,
\[h_{inv}(K,Q;\Sigma_1)\geq h_{inv}(\pi(K),\pi(Q);\Sigma_2).\]
\end{prop}

\begin{proof}
Suppose $\scrS\subset U_1^n$ is an $(n,K,Q)$-spanning set of $(K,Q)$. Let \[r(\scrS)=\{r(\w_0)r(\w_1)\cdots r(\w_{n-1}):\w\in\scrS\}.\]
Then $r(\scrS)\subset U_2^{n}$. We shall show that $r(\scrS)$ is an $(n,\pi(K),\pi(Q))$-spanning set of $(\pi(K),\pi(Q))$. To this end, fix $y\in\pi(K)$. Thus there exists $x\in K$ with $\pi(x)=y$. Since $\scrS$ is an $(n,K,Q)$-spanning set, there exists $S\subset\scrS$ such that $S\in AF^n(Q)$ and $x\in Q_S$. We will find a subset $\hat{S}\subset r(S)$ such that $\hat{S}\in AF^n(\pi(Q))$ and $y\in \pi(Q)_{\hat{S}}$. To see this, let
\[S_{m}=\{s\in U^{m+1}:\exists~\omega\in S~\text{s.t.}~\omega_{[0,m]}=s\},~m\in\{0,1,\ldots,n-1\}.\]
Then we have $S_0=\{s_0\}$, where $s_0$ is the common initial symbol for every $\omega\in S$, and
\[
F_1(x,s_0)\subset\bigcup_{\substack{s\in S_1}}Q_{s_1},~~
F_1(x,s_0)\cap Q_{s_1}\neq\emptyset,~\forall~s\in S_1.
\]
Since $(\pi,r)$ is a semi-conjugacy from $\Sigma_1$ to $\Sigma_2$,
\[F_2(y,r(s_0))=F_2(\pi(x),r(s_0))\subset\pi(F_1(x,s_0))\subset\pi(Q).\]
Let $\hat{S_0}=r(S_0)$. Then $\hat{S}_0\in AF^1(\pi(Q))$ and $y\in \pi(Q)_{\hat{S}_0}$.
Let $B_{r(s_0)}=F_2(\pi(x),r(s_0))$. Then for every $z\in B_{r(s_0)}$ there exists $\hat{z}\in F_1(x,s_0)$ such that $\pi(\hat{z})=z$. Denote the set of all these points by $A_{s_0}$. Thus $\pi(A_{s_0})=B_{r(s_0)}$ and
\[A_{s_0}\subset F_1(x,s_0)\subset\bigcup_{\substack{s\in S_1}}Q_{s_1}.\]
Let
\[A[s_0]=\{s\in S_1:A_{s_0}\cap Q_{s_1}\neq\emptyset\}.\]
Set $\hat{S}_1=r(A[s_0])$. Thus for any $\hat{s}\in\hat{S}_1$,
\[B_{r(s_0)}\subset \cup_{\hat{s}\in\hat{S}_1}\pi(Q)_{\hat{s}_1}~\text{and}~
B_{r(s_0)}\cap\pi(Q)_{\hat{s}_1}\neq\emptyset\]
and
\[F_2(B_{r(s_0)}\cap\pi(Q)_{\hat{s}_1},\hat{s}_1)\subset\pi(F_1(A_{s_0}\cap Q_{s_1},s_1))\subset\pi(Q).\]
Then $\hat{S}_1\in AF^2(\pi(Q))$ and $y\in \pi(Q)_{\hat{S}_1}$.

Repeating this process, we can find the desired $\hat{S}\in AF^n(\pi(Q))$ and $y\in \pi(Q)_{\hat{S}}$. Hence, $r(\scrS)$ is an $(n,\pi(K),\pi(Q))$ of $(\pi(K),\pi(Q))$. This completes the proof.
\end{proof}

\subsection{Invariance feedback entropy}\label{subsec:ife}
Let us recall the concept of invariance feedback entropy proposed by Tomar et al.~\cite{Tomar2017invariance}. Assume that $\Sigma= (X, U, F)$ is a system and $Q\subset X$ is controlled invariant. A pair $(\mathcal{A},G)$ is called an \emph{invariant cover} of $Q$
if $\mathcal{A}$ is a finite cover of $Q$ and $G$ is a map $G:\mathcal{A}\to U$ such that for every $A\in\mathcal{A}$ we have $F(A,G(A))\subset Q$.

Suppose $(\mathcal{A},G)$ is an invariant cover of $Q$; and let $n\in\N$ and  $\mathcal{S}\subset\mathcal{A}^n$ be a set of sequences in $\mathcal{A}$. For $\alpha\in \mathcal{S}$ and $t\in [0,n-1)$ we define
\[P(\alpha|_{[0,t]}):=\{A\in\mathcal{A}|\exists\hat{\alpha}\in\mathcal{S}~\text{s.t.}~
\hat{\alpha}|_{[0,t]}=\alpha|_{[0,t]}~\text{and}~A=\hat{\alpha}_{t+1}\}.\]
The set $P(\alpha|_{[0,t]})$ contains the cover elements $A$ so that the sequence $\alpha|_{[0,t]}A$ can be extended to a sequence in $\mathcal{S}$.
If $t=n-1$ then $\alpha|_{[0,n-1]}=\alpha$ and define
\[P(\alpha):=\left\{A\in\mathcal{A}|\exists\hat{\alpha}\in\mathcal{S}~\text{s.t.}~ A=\hat{\alpha}_0\right\},\]
which is actually independent of $\alpha\in\mathcal{S}$ and corresponds to the ``initial" cover elements $A$ in $\mathcal{S}$, i.e., there exists $\alpha\in\mathcal{S}$ with $A=\alpha(0)$.  A set $\mathcal{S}\subset\mathcal{A}^n$ is called \emph{$(n,Q)$-spanning} in $(\mathcal{A},G)$ if
\begin{itemize}
  \item[(1).] the set $P(\alpha)$ with $\alpha\in\mathcal{S}$ covers $Q$;
  \item[(2).] for every $\alpha\in\calS$ and $t\in[0,n-1)$, we have
\[F(\alpha(t), G(\alpha(t))) \subseteq \bigcup_{A^{\prime} \in P(\alpha|_{[0, t]})} A^{\prime}.\]
\end{itemize}
The \emph{expansion number} $N(\mathcal{S})$ associated with $\mathcal{S}$  is defined by
\[N(\mathcal{S}):=\max _{\alpha \in \mathcal{S}} \prod_{t=0}^{n-1} \sharp P\left(\left.\alpha\right|_{[0,t]}\right).\]
Let
\[r_{inv}(n,Q,\mathcal{A},G):=\min\{N(\mathcal{S})|\mathcal{S}~\text{is}~(n,Q)\text{-spanning in} (\mathcal{A},Q)\}.\]
Since $\log r_{inv}(\cdot,Q,\mathcal{A},G)$ is subadditive (see Lemma 1 in~\cite{Tomar2017invariance}), the following limit exits
\[h_{inv}(\mathcal{A},G):=\lim_{n\to\infty}\frac{1}{n}\log r_{inv}(n,Q,\mathcal{A},G),\]
and it is called \emph{entropy} of $(\mathcal{A},G)$.
The \emph{invariance feedback entropy }of $Q$ is defined as
\[h_{inv}^{fb}(Q):=\inf _{(\mathcal{A}, G)} h_{inv}(\mathcal{A}, G),\]
where the infimum is taken over all invariant covers of $Q$.
The following theorem states that invariance entropy is bounded above by invariance feedback entropy.

\begin{thm}\label{thm:inv-less-ifb}
If $\Sigma=(X,U,F)$ is a system and $Q\subset X$ is a controlled invariant set, then
\[h_{inv}(Q)\leq h_{inv}^{fb}(Q).\]
\end{thm}

\begin{proof}
Suppose that $(\mathcal{A},G)$ is an invariant cover of $Q$, $n\in\N$, and $\mathcal{S}\subset\mathcal{A}^n$ is $(n,Q)$-spanning in $(\mathcal{A},Q)$. Let $\mathscr{S}=\{G(\alpha)|\alpha\in\mathcal{S}\}$,
where $G(\alpha):=G(\alpha_0)\cdots G(\alpha_{n-1})$.
It is obvious that
\[Q\subset \bigcup_{S\subset\mathscr{S}~\text{and}~S\in AF^n(Q)}Q_S.\]
Hence, $\mathscr{S}$ is an $(n,Q)$-spanning set of $Q$ and $\sharp\mathscr{S}\leq\sharp\mathcal{S}$. It follows that $r_{inv}(n,Q)\leq\sharp\mathcal{S}$. Applying Lemma 2 in~\cite{Tomar2017invariance}, we have $r_{inv}(n,Q)\leq N(\mathcal{S})$, which implies that $r_{inv}(n,Q)\leq r_{inv}(n,Q,\mathcal{A},G)$. Thus we obtain the desired inequality.
\end{proof}

The following two examples illustrate that both $h_{inv}(Q)<h_{inv}^{fb}(Q)$
and $h_{inv}(Q)=h_{inv}^{fb}(Q)$ may be possible.

\begin{exam}
Let $\Sigma=(X,U,F)$ be a system, where $X=\{0,1,2\}$ and $U=\{a,b\}$. The transition function $F$ is illustrated by

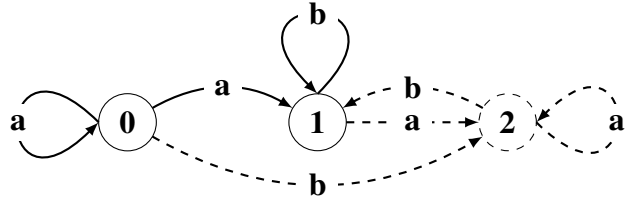
\begin{center}
\begin{tikzpicture}[scale=0.5]
\tikzset{
  LabelStyle/.style = {text = black, font = \bfseries,fill=white },
  VertexStyle/.style = { fill=white, draw=black,circle,font = \bfseries},
  EdgeStyle/.style = {-latex, bend left} }
  \SetGraphUnit{5}
  \Vertex{1}
  \WE(1){0}
  \tikzset{VertexStyle/.style = { fill=white, draw=black,circle,font = \bfseries,dashed}}
  \EA(1){2}
  \Edge[label = a](0)(1)
  \Loop[dist = 4cm, dir = NO, label = a,style={-latex,thick}](0.west)%thick 是edge 的默认线宽，见tkz-graph宏包文档
  \tikzset{EdgeStyle/.style = {-latex}}
  \Loop[dist = 4cm,dir=NO,label = b,style={-latex,thick}](1.north)
  \tikzset{EdgeStyle/.style = {-latex, bend right,dashed}}
  \Edge[label = b](0)(2)
  \tikzset{EdgeStyle/.style = {-latex, bend right,dashed}}
  \Edge[label = b](2)(1)
  \tikzset{EdgeStyle/.style = {-latex,dashed}}
  \Edge[label = a](1)(2)
  \tikzset{EdgeStyle/.style = {bend right,dashed}}
  \Loop[dist = 4cm,dir=SO,label = a,style={-latex,thick}](2.east)
\end{tikzpicture}
\end{center}
The set of interest is $Q:=\{0,1\}$. Then $h_{inv}(Q)=0$ and $h_{inv}^{fb}(Q)=1$.
\end{exam}

\begin{proof}
Let $\mathscr{S}=\{a^ib^{n-1-i}|i=0,1,\ldots,n-1\}$. It is not difficult to check that $\mathscr{S}$ is an $(n,Q)$-spanning set of $Q$. So $h_{inv}(Q)=0$. Put $\mathcal{A}=\{\{0\},\{1\}\}$, and define $G:\mathcal{A}\to U$ by $G(\{0\})=a$ and $G(\{1\})=b$. We shall show that $h_{inv}^{fb}(Q)=1$. Suppose that $\mathcal{S}\subset\mathcal{A}^{n}$ is $(n,Q)$-spanning in $(\mathcal{A},G)$. Then $\alpha=\underbrace{\{0\}\ldots\{0\}}_{n}\in\mathcal{S}$. This yields that $N(\mathcal{S})=2^n$ and
\[r_{inv}(n,Q,\mathcal{A},G)=2^n.\]
It then follows that $h(\mathcal{A},G)=1$. Since $(\mathcal{A},G)$ is the only invariant cover of $Q$, we obtain $h_{inv}^{fb}(Q)=1$.
\end{proof}

\begin{exam}\label{exa:inv-equas-IFE}
Let $\Sigma=(X,U,F)$ be a system, where $X=\{0,1,2\}$ and $U=\{a,b\}$. The transition function $F$ is illustrated by

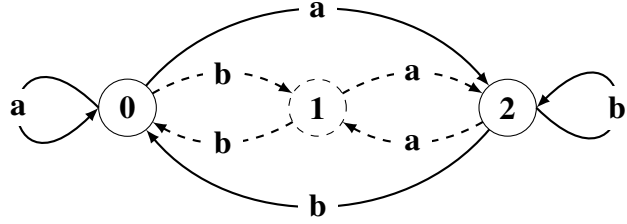
\begin{center}
\begin{tikzpicture}[scale=0.5]
\tikzset{
  LabelStyle/.style = {text = black, font = \bfseries,fill=white },
  VertexStyle/.style = { fill=white, draw=black,circle,font = \bfseries},
  EdgeStyle/.style = {-latex, bend left} }
  \SetGraphUnit{5}
  \tikzset{VertexStyle/.style = { fill=white, draw=black,circle,font = \bfseries,dashed}}
  \Vertex{1}
  \tikzset{VertexStyle/.style = { fill=white, draw=black,circle,font = \bfseries}}
  \EA(1){2}
  \WE(1){0}
  \tikzset{EdgeStyle/.style = {-latex, bend left,dashed}}
  \Edge[label = b](0)(1)
  \Edge[label = b](1)(0)
  \tikzset{EdgeStyle/.style = {bend left}}
  \Loop[dist = 4cm, dir = NO, label = a,style={-latex,thick}](0.west)
  %\tikzset{EdgeStyle/.style = {-latex}}
  %\Loop[dist = 4cm,dir=NO,label = b](1.north)
  \tikzset{EdgeStyle/.style = {-latex, bend left=50}}
  \Edge[label = a](0)(2)
  \tikzset{EdgeStyle/.style = {-latex, bend left,dashed}}
  \Edge[label = a](2)(1)
  \tikzset{EdgeStyle/.style = {-latex,bend left,dashed}}
  \Edge[label = a](1)(2)
  \tikzset{EdgeStyle/.style = {bend right}}
  \Loop[dist = 4cm,dir=SO,label = b,style={-latex,thick}](2.east)
  \tikzset{EdgeStyle/.style = {-latex, bend left=50}}
  \Edge[label = b](2)(0)
\end{tikzpicture}
\end{center}
The set of interest is $Q:=\{0,2\}$. Then $h_{inv}(Q)=h_{inv}^{fb}(Q)=1$.
\end{exam}

\begin{proof}
We see that $h_{inv}^{fb}(Q)=1$ from Example 1 in~\cite{Tomar2017invariance}. We shall show that $h_{inv}(Q)=1$. Suppose that $\mathscr{S}\subset U^n$ is an $(n,Q)$-spanning set. Since $Q_a=\{0\}$, $Q_b=\{2\}$, $F(0,a)=F(2,b)=Q=\{0,2\}$, we have $U^n\subset\mathscr{S}$. It follows that $\sharp\mathscr{S}=2^n$. Hence $r_{inv}(n,Q)=2^n$ and $h_{inv}(Q)=1$.
\end{proof}

\section{Calculations for invariance entropy and IFE}\label{sec:calculation}
This section deals with the calculations for invariance entropy and invariance feedback entropy.

\subsection{Calculation for entropy of quasi-invariant-partitions}
Let $\Sigma=(X,U,F)$ be a system, $Q\subset X$ a controlled invariant set, and $(\calA,G)$ an invariant cover of $Q$. Before going further, we borrow some notations from~\cite{Tomar2017invariance} and introduce some new concepts.
For every $A\in\calA$, let $D_\calA(A):=\{A'\in\calA:F(A,G(A))\cap A'\neq\emptyset\}$ and $w_\calA(A):=\log \sharp D_\calA(A)$.
When there is no ambiguity, we write $D(A)$ and $w(A)$ instead of $D_\calA(A)$ and $w_\calA(A)$.
Given $m\in\N$, a sequence $(A_i)_{i=0}^m$ is called \emph{admissible} for $(\calA,G)$ if $F(A_i,G(A_i))\cap A_{i+1}\neq\emptyset$ for every $0\leq i< m$. Set
\begin{align*}
W_m(\calA,G):=\{(A_i)_{i=0}^{m-1}|(A_i)_{i=0}^{m-1}~\text{is admissible for}~(\calA,G) \}.
%W_\infty(\calA,G)&:=\{(A_i)_{i=0}^{\infty}:(A_i)_{i=0}^{\infty}~\text{is admissible for}~ (\calA,G) \}.
\end{align*}
A sequence $c=(A_i)_{i=0}^{k-1}$ is called \emph{an irreducible sequence of period $k$} for $(\calA,G)$ if $c^\infty$ is admissible for $(\calA,G)$ and $A_i\neq A_j$ for distinct $i,j$. (By ``$c^{\infty}$'' we mean $ccc\cdots$.) The \emph{period} of $c$ is defined as $k$ (denoted by $l(c)$) and the \emph{mean weight} for $c$ is defined as
\[\overline{w}(c):=\frac{1}{k}\sum_{i=0}^{k-1}w(A_i).\]
The \emph{maximum mean weight} $\overline{w}^*(\calA,G)$ is defined by
$\overline{w}^*(\calA,G):=\max_{c}\overline{w}(c)$,
where the maximum is taken over all irreducible periodic sequences for $(\calA,G)$.

The \emph{adjacency matrix} $M_{\calA,G}=(M_{AB})$ of $(\calA,G)$ is defined by
\[M_{AB}=\left\{
           \begin{array}{ll}
             1, & F(A,G(A))\cap B\neq\emptyset; \\
             0, & otherwise.
           \end{array}
         \right.
\]
We define the \emph{weighted adjacency matrix} $W_{\calA,G}=(W_{AB})$ with $A, B\in\calA$ of $(\calA,G)$ as
\[W_{AB}=\left\{
           \begin{array}{ll}
             \sharp D(A), & F(A,G(A))\cap B\neq\emptyset; \\
             0, & otherwise.
           \end{array}
         \right.
\]
Recall that the $l_\infty$-norm for a $n\times n$ matrix $M$ is defined by
$\|M\|_\infty=\max_{1\leq i,j\leq n}|a_{ij}|$.

An invariant cover $(\calA,G)$ is said to be a \emph{quasi-invariant-partition} of $Q$ if
\begin{equation}\label{eq:quasi-partition-a}
A\setminus\bigcup_{B\in\calA,B\neq A}B\neq\emptyset,~\forall A\in\calA
\end{equation}
and
\begin{equation}\label{eq:quasi-parition-b}
F(A,G(A))\bigcap(B\setminus \bigcup_{C\in D(A),~C\neq B}C)\neq\emptyset,~\forall A\in\calA, B\in D(A);
\end{equation}
an \emph{invariant partition} if $\calA$ is a partition of $Q$. Obviously an invariant partition is a quasi-invariant-partition.

Tomar et al. in~\cite{Tomar2020Numerical} obtained an interesting result for computing entropy of an invariant partition. Here we generalize this result to quasi-invariant-partitions.

\begin{thm}\label{thm:IFE-of-quasi-equals-log-sequences}
Suppose that $\Sigma=(X,U,F)$ is a system, $Q\subset X$ is controlled invariant, and $(\calA,G)$ is a quasi-invariant-partition of $Q$. Then
\[h_{inv}(\calA,G)=\lim_{m\to\infty}\frac{1}{m}\max_{\alpha\in W_m(\calA,G)}\sum_{i=0}^{m-2}w(\alpha(i)).\]
\end{thm}

\begin{proof}
\textbf{Claim 1.} $W_m(\calA,G)$ is an $(m,Q)$-spanning set of $(\calA,G)$ for $m\geq2$.

For every $\alpha\in W_m(\calA,G)$, $P(\alpha)=\{\alpha(0):\alpha\in W_m(\calA,G)\}=\calA$, $P(\alpha|_{[0,t]})=D(\alpha(t))$, and \[F(\alpha(t),G(\alpha(t)))\subset \bigcup_{A'\in D(\alpha(t))}A'.\]
Hence Claim 1 holds.

\textbf{Claim 2.} For any $(m,Q)$-spanning set $\calS$ in $(\calA,G)$, $W_m(\calA,G)\subset \calS$.

We apply the inductive argument to show this claim. Suppose that $\calS$ is a $(2,Q)$-spanning set of $(\calA,G)$. Then $P(\alpha)=\calA$ for every $\alpha\in\calS$ by formula~\eqref{eq:quasi-partition-a}. Thus $W_2(\calA,G)\subset\calS$ follows from formula~\eqref{eq:quasi-partition-a}. So the claim holds for $m=2$. Assume that the claim holds for $2\leq i\leq m-1$. Let $\calS$ be an $(m,Q)$-spanning set of $(\calA,Q)$. Then $\calS'=\{\alpha|_{[0,m-2]}:\alpha\in\calS\}$ is an $(m-1,Q)$-spanning set and $W_{m-1}(\calA,G)\subset\calS'$. Hence $W_m(\calA,G)\subset\calS$ by formula~\eqref{eq:quasi-parition-b}. So the claim holds for every $m\geq2$.

Combining Claim 1 with Claim 2 yields $r_{inv}(m,Q,\calA,G)=N(W_m(\calA,G))$. It follows that
\begin{align*}
h_{inv}(\calA,G)&=\lim_{m\to\infty}\frac{1}{m}\log N(W_m(\calA,G))\\
&=\lim_{m\to\infty}\frac{1}{m}\log \big(\sharp\calA\max_{\alpha\in W_m(\calA,G)}\prod_{i=0}^{m-2}D(\alpha(i))\big)\\
&=\lim_{m\to\infty}\frac{1}{m}\max_{\alpha\in W_m(\calA,G)}\sum_{i=0}^{m-2}w(\alpha(i)).
\end{align*}

\end{proof}

\begin{thm}\label{thm:invariant-partition-log-weight}
Let $\Sigma=(X,U,F)$ be a system, $Q\subset X$ a controlled invariant set, and $(\calA,G)$ a quasi-invariant-partition of $Q$. Then
\[h_{inv}(\calA,G)=\overline{w}^*(\calA,G),\]
where $\overline{w}^*(\calA,G)$ is the maximum mean weight.
\end{thm}

\begin{proof}
We first show that $h_{inv}(\calA,G)\geq \overline{w}^*(\calA,G)$.
%From Theorem~\ref{thm:IFE-of-quasi-equals-log-sequences}, we see
%\[h_{inv}(\calA,G)=\lim_{m\to\infty}\frac{1}{m}\max_{\alpha\in W_m(\calA,G)}\sum_{i=0}^{m-2}w(\alpha(i)).\]
Suppose that $c=(A_i)_{i=0}^{k-1}$ is an irreducible sequence of period $k$. For any $m\in\N$, let
\[\beta_{c,m}:=\underbrace{A_0\cdots A_{k-1}\cdots A_0\cdots A_{k-1}}_{m}A_0.\]
Then $\beta_{c,m}\in W_{mk+1}(\calA,G)$. Utilizing Theorem~\ref{thm:IFE-of-quasi-equals-log-sequences}, we have
\begin{align*}
h_{inv}(\calA,G)&=\lim_{m\to\infty}\frac{1}{mk+1}\max_{\alpha\in W_{mk+1}(\calA,G)}\sum_{i=0}^{mk-1}w(\alpha(i))\\
&\geq \lim_{m\to\infty}\frac{1}{mk+1}\sum_{i=0}^{mk-1}w(\beta_{c,m}(i))\\
&=\lim_{m\to\infty}\frac{m}{mk+1}\sum_{i=0}^{k-1}w(A_i)\\
&=\frac{1}{k}\sum_{i=0}^{k-1}w(A_i)=\overline{w}(c).
\end{align*}
The desired inequality immediately follows from the arbitrariness of $c$.

We now show the inverse inequality. For any $m\geq\sharp\calA+3$ and $\alpha_1\in W_m(\calA,G)$, we have $\alpha_1(0)\alpha_1(1)\cdots\alpha_1(m-2)\in W_{m-1}(\calA,G)$. Using the pigeonhole principle, we can pick an irreducible sequence $c_1=(A_{1,i})_{i=0}^{k_1-1}$ of period $k_1$ in $(\calA,G)$ and  there exists $p_1\in[0,\sharp\calA]$ such that \[\alpha_1(p_1)\alpha_1(p_1+1)\cdots\alpha_1(p_1+k_1)=A_{1,0}A_{1,1}\cdots A_{1,k_1-1}.\]
We thus have
\[w(\alpha_1(p_1))+w(\alpha_1(p_1+1))+\cdots w(\alpha_1(p_1+k_1-1))=k_1\overline{w}(c_1)\leq k_1\overline{w}^*(\calA,G).\]
Let
\[\alpha_2=\alpha_1(0)\cdots\alpha_1(p_1-1)\alpha_1(p_1+k_1)\cdots\alpha_1(m-2).\]
Clearly, $\alpha_2\in W_{m-k_1-1}(\calA,G)$. Applying the pigeonhole principle repeatedly, we can find a sequence of irreducible sequences of period $\{c_j\}_{j=1}^q$ in $(\calA,G)$, a sequence $\{\alpha_{j+1}\}_{j=1}^{q}$ with $\alpha_{j+1}\in W_{m-\sum_{i=1}^jk_i-1}(\calA,G)$ and two sequence numbers $\{k_j\}_{j=1}^q$ with $l(c_j)=k_j$ and $\{p_j\}_{j=1}^q$ with $p_j\in[0,\sharp\calA]$ such that
\[\alpha_j(p_j)\alpha_j(p_j+1)\cdots\alpha_j(p_j+k_j)=A_{j,0}A_{j,1}\cdots A_{j,k_j},\]
\[w(\alpha_j(p_j))+w(\alpha_j(p_j+1))+\cdots w(\alpha_j(p_j+k_j-1))=k_j\overline{w}(c_j)\leq k_j\overline{w}^*(\calA,G),\]
and $\alpha_{q+1}\in W_{m-\sum_{j=1}^qk_j-1}(\calA,G)$,
where $m-\sum_{j=1}^qk_j-1\in[0,\sharp\calA]$.
It is convenient to write $L=m-\sum_{j=1}^qk_j-1$. Then
\begin{align*}
\sum_{i=0}^{m-2}w(\alpha(i))
&=\sum_{j=1}^q\sum_{i=0}^{k_j-1}w(\alpha_j(p_j+i))+\sum_{i=0}^{L-1}w(\alpha_{q+1}(i))\\
&\leq\sum_{j=1}^qk_j\overline{w}^*(\calA,G)+L\max_{A\in\calA}w(A)\\
&=(m-1-L)\overline{w}^*(\calA,G)+L\max_{A\in\calA}w(A)\\
&\leq(m-1)\overline{w}^*(\calA,G)+\sharp\calA\max_{A\in\calA}w(A).
\end{align*}
Therefore,
\begin{align*}
h_{inv}(\calA,G)&=\lim_{m\to\infty}\frac{1}{m}\max_{\alpha\in W_m(\calA,G)}\sum_{i=0}^{m-2}w(\alpha(i))\\
&\leq\lim_{m\to\infty}\frac{m-1}{m}\overline{w}^*(\calA,G)
+\lim_{m\to\infty}\frac{1}{m}\sharp\calA\max_{A\in\calA}w(A)\\
&=\overline{w}^*(\calA,G).
\end{align*}
This completes the proof.
\end{proof}

\begin{rem}
Passing $(\calA,G)$ to an invariant partition of $Q$, we recover Theorem 1 in~\cite{Tomar2020Numerical}.
\end{rem}

\begin{cor}\label{cor:quasi-par-bigger-equal-par}
Let $\Sigma=(X,U,F)$ be a system and $Q\subset X$ be controlled invariant. For every quasi-invariant-partition $(\calA,G)$ of $Q$, there exists an invariant partition $(\calA',G')$ of $Q$ such that $\sharp\calA=\sharp\calA'$ and $h_{inv}(\calA',G')\leq h_{inv}(\calA,G)$.
\end{cor}

\begin{proof}
Let $\calA=\{A_1,\ldots,A_p\}$. Define sets $A_1',\ldots,A_p'$ by $A_1'=A_1$, $A_j'=A_j\setminus\cup_{i=1}^{j-1}A_i$, for any $2\leq j\leq p$, and $G'(A_j'):=G(A_j)$, $j=1,\ldots,p$. Then $(\calA',G')$ is an invariant partition of $Q$. Suppose $c'=(A_i')_{1}^{q-1}$ is an irreducible periodic sequence in $(\calA',G')$ so that $h_{inv}(\calA',G')=\overline{w}(c')$. Since $A_i'\subset A_i$ for any $1\leq i\leq p$, it follows that $c:=(A_i)_{1}^{q-1}$ is an irreducible periodic sequence in $(\calA,G)$ and $\overline{w}(c')\leq\overline{w}(c)$; hence, we have by Theorem~\ref{thm:invariant-partition-log-weight} $h_{inv}(\calA,G)\geq\overline{w}(c)\geq h_{inv}(\calA',G')$.
\end{proof}

In the following example, we construct a system that has a quasi-invariant-partition $(\calA_1,G_1)$, where we find two invariant partitions $(\calA_2,G_2)$ and $(\calA_3,G_3)$ such that
\[h_{inv}(\calA_2,G_2)= h_{inv}(\calA_1,G_1)~\text{and}~h_{inv}(\calA_3,G_3)< h_{inv}(\calA_1,G_1).\]

\begin{exam}
Let $\Sigma=(X,U,F)$ be a system, where $X=\{0,1,2,3\}$ and $U=\{a,b\}$. The transition function $F$ is illustrated by
\begin{figure}[H]
  \centering
\begin{tikzpicture}[scale=0.5]
\tikzset{
  LabelStyle/.style = {text = black, font = \bfseries,fill=white },
  VertexStyle/.style = { fill=white, draw=black,circle,font = \bfseries},
  EdgeStyle/.style = {-latex, bend left} }
  \SetGraphUnit{5}
  \Vertex{1}
  \WE(1){0}
  \EA(1){2}
  \tikzset{VertexStyle/.style = { fill=white, draw=black,circle,font = \bfseries,dashed}}
  \EA(2){3}
  \tikzset{EdgeStyle/.style = {-latex, bend left=50}}
  \Edge[label = a](0)(2)
  \tikzset{EdgeStyle/.style = {-latex, bend left}}
  \Edge[label = a](0)(1)
  \Edge[label = b](2)(0)
  \tikzset{EdgeStyle/.style = {-latex}}
  \Edge[label = a](1)(0)
  \Edge[label = b](1)(2)
  \tikzset{EdgeStyle/.style = {-latex,dashed}}
  \Edge[label = a](2)(3)
  \tikzset{EdgeStyle/.style = {-latex,bend right,dashed}}
  \Edge[label = b](0.south)(3)
  \Loop[dist = 4cm,dir=SO,label ={a,b},style={-latex,thick}](3.east)
%  \tikzset{EdgeStyle/.style = {-latex, bend right,dashed}}
%  \Edge[label = b](0)(2)
%  \tikzset{EdgeStyle/.style = {-latex, bend right,dashed}}
%  \Edge[label = b](2)(1)
%  \tikzset{EdgeStyle/.style = {-latex,dashed}}
%  \Edge[label = a](1)(2)
%  \tikzset{EdgeStyle/.style = {bend right,dashed}}
%  \Loop[dist = 4cm,dir=SO,label = a,style={-latex,thick}](2.east)
\end{tikzpicture}
 \caption{}\label{fig:inv-par-less-equal-quasi}
\end{figure}
The set of interest is $Q:=\{0,1,2\}$. Let $A_{11}=\{0,1\}$, $A_{12}=\{1,2\}$, $A_{21}=\{2\}$, $A_{31}=\{0\}$, $A_{32}=\{1\}$, and $A_{33}=\{2\}$; and set $\calA_1=\{A_{11},A_{12}\}$, $\calA_2=\{A_{11},A_{21}\}$, and $\calA_3=\{A_{31},A_{32},A_{33}\}$. Define $G_1:\calA_1\to U$ by $G_1(A_{11})=a$ and $G_1(A_{12})=b$; $G_2:\calA_2\to U$ by $G_2(A_{11})=a$ and $G_2(A_{21})=b$; and $G_3:\calA_3\to U$ by $G_3(A_{31})=a$, $G_3(A_{32})$=a, and $G_3(A_{33})=b$. Then
\[h_{inv}(\calA_1,G_1)=h_{inv}(\calA_2,G_2)=1,~h_{inv}(\calA_3,G_3)=\frac12.\]
\end{exam}

\begin{proof}
By definition, $(\calA_1,G_1)$ is a quasi-invariant-partition, and $(\calA_2,G_2)$ and $(\calA_3,G_3)$ are invariant partitions.
We now compute entropy for $(\calA_1,G_1)$, $(\calA_2,G_2)$, and $(\calA_3,G_3)$.
%, it is convenient to utilize their weighted adjacency matrices:
%\begin{equation*}
%W_{\calA_1,G_1}=\begin{pNiceMatrix}[first-row,last-col,nullify-dots]
%\{0,1\} & \{1,2\} &   \\
%2       & 2       & \{0,1\} \\
%2       & 2       & \{1,2\}\\
%\end{pNiceMatrix},~
%W_{\calA_2,G_2}=\begin{pNiceMatrix}[first-row,last-col,nullify-dots]
%\{0,1\} & \{2\}   &   \\
%2       & 2       & \{0,1\} \\
%1       & 0       & \{2\}\\
%\end{pNiceMatrix},
%\end{equation*}
%\begin{equation*}
%W_{\calA_3,G_3}=\begin{pNiceMatrix}[first-row,last-col,nullify-dots]
%\{0\} & \{1\} & \{2\} &        \\
%0     & 2     &  2    & \{0\}  \\
%1     & 0     &  0    & \{1\}  \\
%1     & 0     &  0    & \{2\}  \\
%\end{pNiceMatrix}.
%\end{equation*}
From Fig.~\ref{fig:inv-par-less-equal-quasi}, we see $A_{11}=\{0,1\}$ is an irreducible sequence of period $1$ for both $(\calA_1,G_1)$ and $(\calA_2,G_2)$ and $\overline{w}(A_{11})=1$. Since $h_{inv}(\calA_1,G_1)\leq1$ and $h_{inv}(\calA_2,G_2)\leq1$, it follows from Theorem~\ref{thm:invariant-partition-log-weight} that $h_{inv}(\calA_1,G_1)=1$ and $h_{inv}(\calA_2,G_2)=1$. Fig.~\ref{fig:inv-par-less-equal-quasi} also tells us that $(\calA_3,G_3)$ only has irreducible sequences of period $2$: $A_{31}A_{32}$, $A_{31}A_{33}$, $A_{32}A_{31}$, and $A_{33}A_{31}$. Applying Theorem~\ref{thm:invariant-partition-log-weight}, an easy computation shows that $h_{inv}(\calA_3,G_3)=\frac12$.
\end{proof}

Theorem~\ref{thm:invariant-partition-log-weight} states that the entropy for a quasi-invariant-partition is equal to its maximum mean weight. From a numerical point of view, we shall give lower and upper bounds for entropy of quasi-invariant-partitions. In some special case, we can obtain the entropy for a quasi-invariant-partition by the logarithm
of the spectral radius of its weighted adjacency matrix.

\begin{thm}\label{thm:upper-lower-bounds-partition}
Let $\Sigma=(X,U,F)$ be a system and $Q\subset X$. If $(\calA,G)$ is a quasi-invariant-partition of $Q$, then
\[\log \rho(W_{\calA,G})-\log \rho(M_{\calA,G})\leq h_{inv}(\calA,G)\leq\min\{\log \|W_{\calA,G}\|_\infty,\log \rho(W_{\calA,G})\},\]
where $\rho(W_{\calA,G})$ is the spectral radius of $W_{\calA,G}$ (the maximum of absolute values of its eigenvalues).
Particularly, if $\rho(M_{\calA,G})=1$, then $h_{inv}(\calA,G)=\log \rho(W_{\calA,G})$.
\end{thm}

\begin{proof}
We first show the right hand side inequality. It is clear that $h_{inv}(\calA,G)\leq\log \|W_{\calA,G}\|_\infty$. Since $(\calA,G)$ is a quasi-invariant-partition, it follows from the proof of Theorem~\ref{thm:IFE-of-quasi-equals-log-sequences} that
\[r_{inv}(n,Q,\mathcal{A},G)=\sharp\calA\cdot\max_{\alpha\in W_n(\calA,G)}\prod_{i=0}^{n-2}\sharp D(\alpha_i).\]
%Suppose $\calS$ is an $(n,Q)$-spanning set of $(\calA,G)$. We can pick $\alpha\in\calS$ such that
%\begin{align*}
%N(\calS)&=\max _{\beta \in \mathcal{S}} \prod_{t=0}^{n-1} \sharp P\left(\left.\beta\right|_{[0,t]}\right)\\
%&=\prod_{t=0}^{n-1} \sharp P\left(\left.\alpha\right|_{[0,t]}\right)\\
%&\leq\prod_{t=0}^{n-2} \sharp D(\alpha_t)\sharp\calA.
%\end{align*}
For any $\alpha\in W_n(\calA,G)$, we have
\[\prod_{t=0}^{n-2}\sharp D(\alpha_t)=W_{\alpha_0\alpha_1}\cdot W_{\alpha_1\alpha_2}\cdots W_{\alpha_{n-2}\alpha_{n-1}}\leq \|W_{\calA,G}^{n-1}\|_\infty.\]
This implies that
\[r_{inv}(n,Q,\mathcal{A},G)\leq\sharp\calA\cdot\|W_{\calA,G}^{n-1}\|_\infty.\]
Hence,
\begin{align*}
h_{inv}(\calA,G)&=\lim_{n\to\infty}\frac{1}{n}\log r_{inv}(n,Q,\calA,G)\\
&\leq\lim_{n\to\infty}\frac{1}{n}\log \sharp\calA\cdot\|W_{\calA,G}^{n-1}\|_\infty\\
&=\lim_{n\to\infty}\frac{n-1}{n}\log \|W_{\calA,G}^{n-1}\|_\infty^{\frac1{n-1}}.
\end{align*}
Employing Theorem 5.7.10 in~\cite{Horn-Johnson2013}, we obtain
$h_{inv}(\calA,Q)\leq\log \rho(W_{\calA,G})$.

We now show the left hand side inequality holds.
For any $\beta\in W_n(\calA,G)$, we have
\[\prod_{t=0}^{n-2}\sharp D(\beta_i)=W_{\beta_0\beta_1}\cdot W_{\beta_1\beta_2}\cdots W_{\beta_{n-2}\beta_{n-1}}\]
and
\[M_{\beta_0\beta_1}\cdot M_{\beta_1\beta_2}\cdots M_{\beta_{n-2}\beta_{n-1}}=1.\]
Then
\[W_{\beta_0\beta_{n-1}}\leq M_{\beta_0\beta_{n-1}}\cdot\max_{\alpha\in W_n(\calA,G)}\prod_{i=0}^{n-2}\sharp D(\alpha_i)\leq \|M_{\calA,G}^{n-1}\|_\infty\cdot\max_{\alpha\in W_n(\calA,G)}\prod_{i=0}^{n-2}\sharp D(\alpha_i).\]
It follows that
\[\|W_{\calA,G}^{n-1}\|_\infty\leq \|M_{\calA,G}^{n-1}\|_\infty\cdot\max_{\alpha\in W_n(\calA,G)}\prod_{i=0}^{n-2}\sharp D(\alpha_i).\]
Thus
\begin{align*}
\lim_{n\to\infty}\frac{1}{n}\log r_{inv}(n,Q,\mathcal{A},G)&\geq\lim_{n\to\infty}\frac{1}{n}\log \sharp\calA\cdot\frac{\|W_{\calA,G}^{n-1}\|_\infty}{\|M_{\calA,G}^{n-1}\|_\infty}\\
&=\log \rho(W_{\calA,G})-\log \rho(M_{\calA,G}).
\end{align*}
\end{proof}

\begin{rem}%(1) We see from the proof of Theorem~\ref{thm:upper-lower-bounds-partition} that the right hand side inequality holds for every invariant cover.
(1) Since the norm $\|\cdot\|_\infty$ is not spectrally dominant, that is, there exists a matrix $M$ such that $\|M\|_\infty<\rho(M)$, we take the minimum of $\log \|W_{\calA,G}\|_\infty$ and $\log \rho(W_{\calA,G})$ in the right hand side of the inequality in Theorem~\ref{thm:upper-lower-bounds-partition}. See Exa.~\ref{exa:inv-fb-exa}.

(2) If $(\calA,G)$ is a quasi-invariant-partition of $Q$, then
\begin{itemize}
  \item[(\romannumeral1)]$\rho(M_{\calA,G})\geq1$, which implies that $\log \rho(M_{\calA,G})\geq0$;
      \begin{itemize}
        \item Since for any $A\in\calA$ there exists $B\in\calA$ so that $M_{AB}=1$, we have
            \[\|M_{\calA,G}^n\|_\infty\geq1,~\forall~n\in\N.\]
            Then $\rho(M_{\calA,G})\geq1$.
      \end{itemize}
  \item[(\romannumeral2)]$\rho(W_{\calA,G})\geq\rho(M_{\calA,G})$, which means that $\log \rho(W_{\calA,G})-\log \rho(M_{\calA,G})\geq0$.
      \begin{itemize}
        \item Since $W_{AB}\geq M_{AB}$, $W_{AB}=\sharp D(A)M_{AB}$, and $M_{AB}=1$ for any $A,B\in\calA$, we have
            \[\|W_{\calA,G}^n\|_\infty\geq \min_{A\in\calA}\{(\sharp D(A))^{n-1}\}\cdot\|M_{\calA,G}^n\|_\infty,~\forall~n\in\N.\]
            It therefore follows that
            \[\log \rho(W_{\calA,G})-\log \rho(M_{\calA,G})\geq \min_{A\in\calA}\{w(A)\}.\]
      \end{itemize}
        \item[(\romannumeral3)] If $\rho(M_{\calA,G})>1$, then we can use Theorem~\ref{thm:invariant-partition-log-weight} to compute the entropy of $(\calA,G)$. On the other hand, if $\rho(M_{\calA,G})=1$, then the entropy of $(\calA,G)$ is $\log \rho(W_{\calA,G})$. See Exa.~\ref{exa:inv-fb-exa}.
\end{itemize}

\end{rem}

\subsection{Calculation of invariance entropy and IFE for some control systems}

Let $\Sigma=(X,U,F)$ be a system, $Q\subset X$, and $V\subset U$. We say that $V$ is a \emph{cover} of $Q$ if $Q\subset \cup_{a\in V}Q_a$, where $Q_a=\{x\in Q| F(x,a)\subset Q\}$. The \emph{admissible matrix} $M_{Q,V}=(M_{ab})_{a,b\in V}$ of $Q$ with respect to $V$ is defined by
\[M_{ab}=\left\{
           \begin{array}{ll}
             1, & \exists~x\in Q_a~s.t.~F(x,a)\cap Q_b\neq\emptyset; \\
             0, & otherwise.
           \end{array}
         \right.
\]

Recall that the $l_1$ norm of an $n\times n$ matrix $M$ is
$\|M\|_1=\sum_{i,j=1}^n|a_{ij}|$.

\begin{prop}
Let $\Sigma=(X,U,F)$ be a system, $Q\subset X$ a controlled invariant set, and $V\subset U$ a cover of $Q$. Then
$h_{inv}(Q)\leq \log \rho(M_{Q,V})$.
\end{prop}

\begin{proof}
Since $V$ is a cover of $Q$, we can pick an $(n,Q)$-spanning set $\scrS\subset V^{n}$ for every $n\geq2$. For every $u\in\scrS$, we have $M_{u_0u_1}\cdot M_{u_1u_2}\cdots M_{u_{n-2}u_{n-1}}=1$. Thus $r_{inv}(n,Q)\leq\sharp\scrS\leq \|M_{Q,V}^{n-1}\|_1$. This implies that
\begin{align*}
h_{inv}(Q)&=\limsup_{n\to\infty}\frac{1}{n}\log r_{inv}(n,Q)\\
&\leq\limsup_{n\to\infty}\frac{1}{n}\log \|M_{Q,V}^{n-1}\|_1\\
&=\limsup_{n\to\infty}\frac{n-1}{n}\log \|M_{Q,V}^{n-1}\|_1^{\frac1{n-1}}.
\end{align*}
Using Gelfand formula~\cite[Corollary 5.6.14]{Horn-Johnson2013}, it follows that
$h_{inv}(Q)\leq\log \rho(M_{Q,V})$.
\end{proof}

\begin{cor}
Let $\Sigma=(X,U,F)$ be a system and $Q\subset X$ be controlled invariant. Then
\[h_{inv}(Q)\leq\inf_{V\subset U~\text{covers}~Q}\log \rho(M_{Q,V}).\]
\end{cor}

\begin{thm}\label{thm:symbolic-log-radius}
Let $\Sigma=(X,U,F)$ be a system, $Q\subset X$ a controlled invariant set, $V\subset U$ a finite cover of $Q$. If in addition
\begin{itemize}
  \item[(C.1)] $Q_a\cap Q_b=\emptyset$ for distinct $a,b\in V$,
  \item[(C.2)]there exists $K\subset Q_a$ such that $Q_b\subset F(K,a)$ for every $M_{ab}=1$,
  \item[(C.3)] $Q_c=\emptyset$ for every $c\in U\setminus V$.
\end{itemize}
Then
$h_{inv}(Q)=\log \rho(M_{Q,V})$.
\end{thm}

\begin{proof}
To simplify the notation, set $M=M_{Q,V}$. It is easy to see from conditions (C.1) and (C.3) that the set $\scrS_2=\{u_0u_1:M_{u_0u_1}=1\}$ is the only $(2,Q)$-spanning set of $Q$. Assume that
\[\scrS_n=\{u_0u_1\cdots u_{n-1}:M_{u_0u_1}\cdot M_{u_1u_2}\cdots M_{u_{n-2}u_{n-1}}=1\}\]
is the only $(n,Q)$-spanning set of $Q$. We will show that $\scrS_{n+1}=\{u_0u_1\cdots u_{n}:M_{u_0u_1}\cdot M_{u_1u_2}\cdots M_{u_{n-1}u_{n}}=1\}$ is the only $(n+1,Q)$-spanning set of $Q$. Let $\scrS$ be an $(n+1,Q)$-spanning set. Then $\scrS\subset\scrS_{n+1}$ and
\[\scrS|_n=\{w\in U^n:\exists u\in\scrS~s.t.~w_i=u_i,i=0,\ldots,n-1\}\]
is an $(n,Q)$-spanning set. By assumption, we see that $\scrS_n=\scrS|_n$. For every $u\in\scrS_n$, let $B_u:=\{b\in V| M_{u_{n-1}b}=1\}$ and $\scrS_u:=\{ub|b\in B_u\}$.
Condition (C.2) tells us that $\scrS_u\subset\scrS$, and thus $\scrS_{n+1}\subset\scrS$. It immediately follows that $\scrS_{n+1}$ is the only $(n+1,Q)$-spanning set of $Q$ and $r_{inv}(n,Q)=\sharp\scrS_n$ for $n\geq 2$.
A standard induction on $n$ then yields
\begin{align*}
%\sharp\scrS_2&=\sum_{a\in V}\sum_{b \in V} M_{a b},\\
%\sharp\scrS_3&=\sum_{a \in V} \sum_{b \in V} \sum_{c \in V} M_{a b} M_{b c}=\sum_{a \in V} \sum_{c \in V}\left(M^{2}\right)_{a c},\\
\sharp\scrS_{n}&=\sum_{u_{0} \in V} \sum_{u_{n-1} \in V}\left(M^{n-1}\right)_{u_{0}, u_{n-1}}.
\end{align*}
Hence $r_{inv}(n,Q)=\sharp\scrS_n=\|M^{n-1}\|_1$, and thus this with Gelfand formula shows that
$h_{inv}(Q)=\log \rho(M)$.
\end{proof}

\begin{rem}
(1) The formula for invariance entropy in Theorem~\ref{thm:symbolic-log-radius} is completely analogous to that for topological entropy of Markov subshifts (see for example Theorem 3.48 in~\cite{PK2000}). The latter gives us a motivation for the former since a Markov subshift can be described by a binary matrix.

(2) Suppose that $V\subset U$ satisfies the conditions of Theorem~\ref{thm:symbolic-log-radius}. Let $\calA_V=\{Q_a:a\in V\}$ and define $G_V:\calA_V\to U$ by $G(Q_a)=a$ for every $a\in V$. Then $(\calA_V,G_V)$ is an invariant partition of $Q$. From Theorem~\ref{thm:symbolic-log-radius}, we see $h_{inv}(Q)=\log \rho(M_{Q,V})$. It is natural to wonder if $h_{inv}^{fb}(Q)=h_{inv}(\calA_V,G_V)$. However, these conditions of  Theorem~\ref{thm:symbolic-log-radius} are not sufficient (see Exa.~\ref{exa:inv-fb-exa}). But the invariance feedback entropy of $Q$ can be determined by this invariant partition.
\end{rem}

Under the conditions of Theorem~\ref{thm:symbolic-log-radius}, we say $(\calB,G_\calB)$ is a \emph{refinement} of $(\calA_V,G_V)$ if $\calB$ is cover of $Q$, every element $B$ of $\calB$ is contained in some element $A$ of $\calA_V$, and $G_\calB(B)=G_V(A)$ for every $B\in\calB$. Let $\scrB(\calA_V,G_V)$ denote all the refinements of $(\calA_V,G_V)$. We call $(\calC,G_\calC)\in\scrB(\calA_V,G_V)$ an \emph{atom refinement} of $(\calA_V,G_V)$ if $\calC=\{\{x\}:x\in Q\}$ and $\sharp(F(x,a)\cap Q_b)$ is at most $1$ for every $a,b\in V$ and $x\in Q_a$. Note that if $(\calA_V,G_V)$ has an atom refinement then it is unique.

\begin{cor}\label{cor:IFE-equi-partitions}
Under the conditions of Theorem~\ref{thm:symbolic-log-radius},
\[h_{inv}^{fb}(Q)=\inf\{h_{inv}(\calB,G_\calB):(\calB,G_\calB)\in\scrB(\calA_V,G_V)\}.\]
\end{cor}

\begin{proof}
Suppose that $(\calB,G_\calB)$ is an invariant cover. By (C.3) in Theorem~\ref{thm:symbolic-log-radius}, For every $B\in\calB$, $G_\calB(B)\in V$. Since $B\subset \cup_{A\in\calA_V}A$, it follows from (C.1) in Theorem~\ref{thm:symbolic-log-radius} that there exists only one element $A\subset \calA_V$ such that $B\subset A$. Hence $G_\calB(B)=G_V(A)$ and it follows that $(\calB,G_\calB)$ is a refinement of $(\calA_V,G_V)$.
\end{proof}

%Recall that a system $\Sigma=(X,U,F)$ is said to be \emph{finite} if $X$ and $U$ are finite~\cite{Tomar2017invariance}.

%For any $(\calB,G_\calB)\in\scrB(\calA_V,G_V)$, let \[D_\calB(A):=\{A'\in\calB:F(A,G_\calB(A))\cap A'\neq\emptyset\}.\]

\begin{thm}\label{thm:IFE-equi-atom-partitions}
Under the conditions of Theorem~\ref{thm:symbolic-log-radius}, if moreover $\sharp Q$ is finite and $(\calC,G_\calC)$ is the atom refinement of $(\calA_V,G_V)$, then
$h_{inv}^{fb}(Q)=h_{inv}(\calC,G_\calC)$.
\end{thm}

\begin{proof}
Since $(\calC,G_\calC)$ is an invariant partition of $Q$, it immediately follows from  Theorem~\ref{thm:invariant-partition-log-weight} that
$h_{inv}(\calC,G_\calC)=\overline{w}^*(\calC,G_\calC)$.
Take an irreducible periodic sequence $c$ in $(\calC,G_\calC)$ such that
$\overline{w}^*(\calC,G_\calC)=\overline{w}(c)$.
We can without loss of generality assume that $c=(C_i)_{i=0}^{k-1}$, where $k\leq\sharp Q$. Fixing $m\in\N$ and a refinement $(\calB,G_\calB)$ of $(\calA,G)$, let
\[\beta_{c,m}:=\underbrace{C_0\cdots C_{k-1}\cdots C_0\cdots C_{k-1}}_{m}C_0\]
and $\calS$ be an $(mk+1,Q)$-spanning set in $(\calB,G_\calB)$. Since $P(\alpha)$ covers $Q$ for any $\alpha\in\calS$ and $\sharp C_0=1$, there exists $\alpha^0\in\calS$ so that $C_0\subset \alpha^0(0)$ and
\[F(C_0,G_{\calB}(C_0))\subset F(\alpha^0(0),G_\calB(\alpha^0(0)))\subset \bigcup_{A' \in P\big(\alpha^0|_{[0,0]}\big)} A'.\]
Thus there exists $\alpha^1\in\calS$ such that $C_0\subset\alpha^1(0), C_1\subset \alpha^1(1)$.
Repeating this process, we can find $\alpha^{mk+1}\in\calS$ such that
\[C_i\subset\alpha^{mk+1}(jk+i),~j=0,\ldots,m-1,
~i=0,\ldots,k-1,~C_0\subset\alpha^{mk+1}(mk).\]
Since $(\calC,G_\calC)$ is the atom refinement,
\begin{equation}\label{eq:DC}
\sharp D_\calC(\{x\})\leq D^*_x
\end{equation}
for any $x\in Q$,
where
\[D^*_x=\min_{\substack{(\calB,G_\calB)\in\scrB(\calA,G)\\A\in\calB\\\{x\}\subset A}}\min\{\sharp\calF:\calF\subset D_\calB(A),~F(A,G_\calB(A))\subset\bigcup_{A'\in\calF}A'\}.\]
Replacing $\{x\}$ in (\ref{eq:DC}) by $C_i$ implies that
\[\sharp D_\calC(C_i)\leq \sharp P(\alpha^{mk+1}|_{[0,jk+i]})=\sharp P(\alpha^{mk+1}|_{[0,jk+i]}),~j=0,\ldots,m-1,
~i=0,\ldots,k-1.\]
Then
\[N(\mathcal{S})\geq\prod_{t=0}^{mk} \sharp P(\alpha^{mk+1}|_{[0,t]})\geq\prod_{t=0}^{mk-1} \sharp P(\alpha^{mk+1}|_{[0,t]})\geq\left(\prod_{i=0}^{k-1}\sharp D_\calC(C_i)\right)^m.\]
Since $\calS$ is arbitrary,
\[r_{inv}(mk+1,Q,\calB,G_\calB)\geq\left(\prod_{i=0}^{k-1}\sharp D_\calC(C_i)\right)^m.\]
It follows that
\begin{align*}
h_{inv}(\calB,G_\calB)&=\lim_{m\to\infty}\frac{1}{mk+1}\log r_{inv}(mk+1,Q,\calB,G_\calB)\\
&\geq \lim_{m\to\infty}\frac{1}{mk+1}\log \left(\prod_{i=0}^{k-1}\sharp D_\calC(C_i)\right)^m\\
&=\overline{w}(c)=h_{inv}(\calC,G_\calC).
\end{align*}
This together with Corollary~\ref{cor:IFE-equi-partitions} yields the desired equality.
\end{proof}

\begin{exam}\label{exa:inv-fb-exa}
Let $\Sigma=(X,U,F)$ be a system, where $X=\{0,1,2,3,4,5\}$ and $U=\{a,b,c\}$. The transition function $F$ is illustrated by Fig.~\ref{fig:partion-not-sufficient}.

\begin{figure}[H]
  \centering
\begin{tikzpicture}[scale=0.4]
\tikzset{
  LabelStyle/.style = {text = black, font = \bfseries,fill=white },
  VertexStyle/.style = { fill=white, draw=black,circle,font = \bfseries},
  EdgeStyle/.style = {-latex, bend left} }
  \SetGraphUnit{5}
  \Vertex{0}
\tikzset{VertexStyle/.style = { fill=white, draw=white,color=white}}
  \SO(0){6}
\tikzset{VertexStyle/.style = { fill=white, draw=black,circle,font = \bfseries}}
  \SO(6){1}
  \EA(0){2}
\tikzset{VertexStyle/.style = { fill=white, draw=white,color=white}}
  \SO(2){7}
\tikzset{VertexStyle/.style = { fill=white, draw=black,circle,font = \bfseries}}
  \SO(7){3}
  \WE(6){4}
\tikzset{VertexStyle/.style = { fill=white, draw=black,circle,font = \bfseries,dashed}}
  \EA(7){5}
  \tikzset{EdgeStyle/.style = {-latex}}
  \Edge[label = a](0)(2)
  \Edge[label = a](1)(3)
  \Edge[label = b](2)(1)
  \Edge[label = b](3)(0)
  \Edge[label = a](0)(4)
  \tikzset{EdgeStyle/.style = {-latex, bend left=70}}
  %\Edge[label = {a}](0)(2)
  %\Edge[label = {b}](1)(4)
  %\Edge[label = {b}](2)(4)
  \tikzset{EdgeStyle/.style = {-latex}}
  %\Edge[label = {b}](2)(1)
  \tikzset{EdgeStyle/.style = {-latex, bend left}}
  %\Edge[label = {b}](3)(0)
  \tikzset{EdgeStyle/.style = {-latex,dashed}}
  \Edge[label = {b,c}](0)(5)
  \Edge[style={pos=.25},label = {b,c}](1)(5)
  \Edge[label = {a,c}](2)(5)
  \Edge[label = {a,c}](3)(5)
  \tikzset{EdgeStyle/.style = {-latex,bend right=9,dashed}}
  \Edge[style={pos=.25},label = {a,b}](4)(5)
%\Loop[dist = 4cm,dir=SO,label = {a,b},style={-latex,thick,dashed}](4.east)
%  \Loop[dist = 4cm, dir = NO, label = a,style={-latex,thick}](0.west)%thick 是edge 的默认线宽，见tkz-graph宏包文档
  \tikzset{EdgeStyle/.style = {-latex}}
  \Loop[dist = 4cm,label = c,style={-latex,thick}](4.west)
  \Loop[dist = 4cm,dir=EA,label ={a,b,c},style={-latex,thick,dashed}](5.east)
%  \tikzset{EdgeStyle/.style = {-latex, bend right,dashed}}
%  \Edge[label = b](0)(2)
%  \tikzset{EdgeStyle/.style = {-latex, bend right,dashed}}
%  \Edge[label = b](2)(1)
%  \tikzset{EdgeStyle/.style = {-latex,dashed}}
%  \Edge[label = a](1)(2)
%  \tikzset{EdgeStyle/.style = {bend right,dashed}}
%  \Loop[dist = 4cm,dir=SO,label = a,style={-latex,thick}](2.east)
\end{tikzpicture}
\caption{}\label{fig:partion-not-sufficient}
\end{figure}

The set of interest is $Q:=\{0,1,2,3,4\}$. Let
\begin{align*}
\calA_1&=\{A_{10},A_{11},A_{12}\},~A_{10}=\{0,1\},~A_{11}=\{2,3\},~A_{12}=\{4\},\\
\calA_2&=\{A_{20},A_{21},A_{22},A_{23},A_{24}\},~A_{20}=\{0\},~A_{21}=\{1\},
~A_{22}=\{2\},~A_{23}=\{3\},~A_{24}=\{4\},\\
\calA_3&=\{A_{30},A_{31},A_{32},A_{33}\},~A_{30}=\{0\},~A_{31}=\{1\},
~A_{32}=\{2,3\},~A_{33}=\{4\}.
%\calA_4&=\{A_{40},A_{41},A_{42},A_{43}\},~A_{40}=\{0,1\},
%~A_{41}=\{2\},~A_{42}=\{3\},~A_{33}=\{4\}.
\end{align*}
Define $G_i:\calA_i\to U$, $i=1,2,3$ by
\begin{align*}
G_1(A_{10})&=a,~G_1(A_{11})=b,~G_1(A_{12})=c,\\
G_2(A_{20})&=a,~G_2(A_{21})=a,~G_2(A_{22})=b,~G_2(A_{23})=b,~G_2(A_{24})=c,\\
G_3(A_{30})&=a,~G_3(A_{31})=a,~G_3(A_{32})=b,~G_3(A_{33})=c.
%G_4(A_{40})&=a,~G_4(A_{41})=b,~G_4(A_{42})=b,~G_4(A_{43})=c.
\end{align*}
Then
\begin{align*}
h_{inv}(Q)=0,~h_{inv}(\calA_1,G_1)=\frac{1}{2},~h_{inv}(\calA_2,G_2)=\frac{1}{4},~ h_{inv}(\calA_3,G_3)=1,~h_{inv}^{fb}(Q)=\frac{1}{4}.
\end{align*}
\end{exam}

\begin{proof}
Clearly $Q_a=\{0,1\}$, $Q_b=\{2,3\}$, $Q_c=\{4\}$, and thus $U$ is a cover of $Q$. It is obvious that conditions (C.1) and (C.3) in Theorem~\ref{thm:symbolic-log-radius} hold. Since $Q_b\subset F(Q_a,a)$, $Q_c\subset F(0,a)$, $Q_a\subset F(Q_b,b)$, and $Q_c\subset F(Q_c,c)$,
condition (C.2) in Theorem~\ref{thm:symbolic-log-radius} holds.
From Fig.~\ref{fig:partion-not-sufficient}, we have
\[M_{Q,U}=\begin{pNiceMatrix}[first-row,last-col,nullify-dots]
a & b & c &   \\
0 & 1 & 1 & a \\
1 & 0 & 0 & b \\
0 & 0 & 1 & c \\
\end{pNiceMatrix}.\]
A brief computation shows that $\rho(M_{Q,U})=1$. It then follows from Theorem~\ref{thm:symbolic-log-radius} that  $h_{inv}(Q)=0$.

We now compute $h_{inv}(\calA_i,G_i)$, $i=1,2,3$. Since
\[M_{\calA_1,G_1}=\begin{pNiceMatrix}[first-row,last-col,nullify-dots]
A_{10} & A_{11} & A_{12} &   \\
0 & 1 & 1 & A_{10} \\
1 & 0 & 0 & A_{11} \\
0 & 0 & 1 & A_{12} \\
\end{pNiceMatrix},~~
W_{\calA_1,G_1}=\begin{pNiceMatrix}[first-row,last-col,nullify-dots]
A_{10} & A_{11} & A_{12} &   \\
0 & 2 & 2 & A_{10} \\
1 & 0 & 0 & A_{11} \\
0 & 0 & 1 & A_{12} \\
\end{pNiceMatrix},
\]

\[M_{\calA_2,G_2}=\begin{pNiceMatrix}[first-row,last-col,nullify-dots]
A_{20} & A_{21} & A_{22} & A_{23} & A_{24} &   \\
0   & 0   & 1   & 0   & 1   & A_{20} \\
0   & 0   & 0   & 1   & 0   & A_{21} \\
0   & 1   & 0   & 0   & 0   & A_{22} \\
1   & 0   & 0   & 0   & 0   & A_{23} \\
0   & 0   & 0   & 0   & 1   & A_{24} \\
\end{pNiceMatrix},~~
W_{\calA_2,G_2}=\begin{pNiceMatrix}[first-row,last-col,nullify-dots]
A_{20} & A_{21} & A_{22} & A_{23} & A_{24} &   \\
0   & 0   & 2   & 0   & 2   & A_{20} \\
0   & 0   & 0   & 1   & 0   & A_{21} \\
0   & 1   & 0   & 0   & 0   & A_{22} \\
1   & 0   & 0   & 0   & 0   & A_{23} \\
0   & 0   & 0   & 0   & 1   & A_{24} \\
\end{pNiceMatrix},
\]

\[M_{\calA_3,G_3}=\begin{pNiceMatrix}[first-row,last-col,nullify-dots]
A_{30} & A_{31} & A_{32} & A_{33} &   \\
0      & 0      & 1      & 1      & A_{30} \\
0      & 0      & 1      & 0      & A_{31} \\
1      & 1      & 0      & 0      & A_{32} \\
0      & 0      & 0      & 1      & A_{33} \\
\end{pNiceMatrix},~~
W_{\calA_3,G_3}=\begin{pNiceMatrix}[first-row,last-col,nullify-dots]
A_{30} & A_{31} & A_{32} & A_{33} &   \\
0      & 0      & 2      & 2      & A_{30} \\
0      & 0      & 1      & 0      & A_{31} \\
2      & 2      & 0      & 0      & A_{32} \\
0      & 0      & 0      & 1      & A_{33} \\
\end{pNiceMatrix},
\]

%\[M_{\calA_4,G_4}=\begin{pNiceMatrix}[first-row,last-col,nullify-dots]
%A_{40} & A_{41} & A_{42} & A_{43} &   \\
%0      & 1      & 1      & 1      & A_{40} \\
%1      & 0      & 0      & 0      & A_{41} \\
%1      & 0      & 0      & 0      & A_{42} \\
%0      & 0      & 0      & 1      & A_{43} \\
%\end{pNiceMatrix},~~
%W_{\calA_4,G_4}=\begin{pNiceMatrix}[first-row,last-col,nullify-dots]
%A_{40} & A_{41} & A_{42} & A_{43} &   \\
%0      & 3      & 3      & 3      & A_{40} \\
%1      & 0      & 0      & 0      & A_{41} \\
%1      & 0      & 0      & 0      & A_{42} \\
%0      & 0      & 0      & 1      & A_{43} \\
%\end{pNiceMatrix},
%\]
it follows by a straightforward calculation that
\[\rho(M_{\calA_1,G_1})=1,~~\rho(W_{\calA_1,G_1})=\sqrt{2},~~\|W_{\calA_1,G_1}\|_\infty=2,\]
\[\rho(M_{\calA_2,G_2})=1,~~\rho(W_{\calA_2,G_2})=\sqrt[4]{2},~~\|W_{\calA_2,G_2}\|_\infty=2,\]
\[\rho(M_{\calA_3,G_3})=\sqrt{2},~~\rho(W_{\calA_3,G_3})=\sqrt{6},~~\|W_{\calA_3,G_3}\|_\infty=2.\]
%\[\rho(M_{\calA_4,G_4})=\sqrt{2},~~\rho(W_{\calA_4,G_4})=\sqrt{6},~~\|W_{\calA_4,G_4}\|_\infty=3.\]
It then follows from Theorem~\ref{thm:upper-lower-bounds-partition} that $h_{inv}(\calA_1,G_1)=\frac{1}{2}$ and $h_{inv}(\calA_2,G_2)=\frac{1}{4}$. Noting that
$\overline{w}^*(\calA_3,G_3)=\overline{w}(c)$, where $c=A_{30}A_{32}$,
we have by Theorem~\ref{thm:invariant-partition-log-weight} $h_{inv}(\calA_3,G_3)=1$.

It is not difficult to check that $(A_2,G_2)$ is the atom refinement, and therefore Theorem~\ref{thm:IFE-equi-atom-partitions} asserts that $h_{inv}^{fb}(Q)=\frac14$.
\end{proof}

%%%%%%%%%%%%%%%%%%%%%%%%%%%%%%%%%%%
%%%%%%%%%%%%%%%%%%%%%%%%%%%%%%%%%%%
\section*{Acknowledgements}
The project was supported by National Nature Science Funds of China (No. 12171492).

%\bibliographystyle{ieeetr}
%\bibliography{../../zxfbibtex/Bibtex20181008}

\end{document}